\documentclass[11pt]{article}
\usepackage{fullpage}
\usepackage[T1]{fontenc}
\usepackage{lmodern}
\usepackage{graphicx}
\usepackage{amsmath,amssymb,amsthm}
\usepackage{microtype}
\usepackage{algorithm}
\usepackage{algorithmic}
\usepackage{subcaption}
\usepackage[group-separator={,}]{siunitx}

\newcommand{\C}{\mathbb{C}}
\newcommand{\F}{\mathbb{F}}
\newcommand{\R}{\mathbb{R}}

\newcommand{\Z}{\mathbb{Z}}

\newcommand{\caL}{\mathcal{L}}
\newcommand{\zeros}{\mathbf{0}}
\newcommand{\ones}{\mathbf{1}}
\newcommand{\bx}{\mathbf{x}}
\newcommand{\by}{\mathbf{y}}
\newcommand{\bu}{\mathbf{u}}
\newcommand{\Wc}{W_{c}}
\newcommand{\Wno}{W_{\bar o}}

\newcommand{\Wco}{W_{co}}
\newcommand{\Wnco}{W_{\bar c o}}
\newcommand{\Wcno}{W_{c\bar o}}
\newcommand{\Wncno}{W_{\bar c\bar o}}
\renewcommand{\Re}{\operatorname{Re}}
\renewcommand{\Im}{\operatorname{Im}}
\newcommand{\Nin}{N^{\mathrm{in}}}
\newcommand{\Nout}{N^{\mathrm{out}}}
\newcommand{\Tc}{T_c}
\newcommand{\To}{T_o}
\DeclareMathOperator{\GL}{GL}
\DeclareMathOperator{\im}{im}

\DeclareMathOperator{\dimv}{\underline{dim}}
\DeclareMathOperator{\Rep}{Rep}
\usepackage{mathtools}

\usepackage{nicematrix}
\usepackage{tikz}
\usetikzlibrary{positioning,graphs}

\usepackage{xcolor}

\newtheorem{theorem}{Theorem}[section]
\newtheorem{lemma}[theorem]{Lemma}
\newtheorem{proposition}[theorem]{Proposition}
\newtheorem{corollary}[theorem]{Corollary}
\theoremstyle{definition}
\newtheorem{definition}[theorem]{Definition}

\newtheorem{remark}[theorem]{Remark}
\newtheorem*{remark*}{Remark}
\newcommand{\finbox}{\hspace*{\fill}$\rule{0.15cm}{0.2cm}$}
\AtEndEnvironment{remark}{\finbox}
\numberwithin{equation}{section}

\usepackage{hyperref}
\title{\bf Quiver Semistability and Structured Kalman Decompositions for Networked Linear Dynamical Systems}
\author{
Kazuo Murota \\
\url{murota@tmu.ac.jp} \\
Faculty of Economics and Business Administration, \\
Tokyo Metropolitan University, \\
Tokyo 192-0397, Japan; \\
The Institute of Statistical Mathematics, \\ 
Tokyo 190-8562, Japan.
\and    
Tasuku Soma \\ 
\url{soma@ism.ac.jp} \\
The Institute of Statistical Mathematics, \\ 
Tokyo 190-8562, Japan.
}
\date{August 30, 2026}
\begin{document}
\maketitle
\begin{abstract}
    We introduce new notions of controllability and observability for networked linear time-invariant (LTI) systems based on $\sigma$-semistability of quiver representations.
    Utilizing King's criterion for $\sigma$-semistability, we define a network generalization of the Kalman decomposition for networked LTI systems, which systematically decomposes the local and interconnection dynamics while respecting the underlying network structure.
    Furthermore, we present efficient algorithms for deciding the proposed controllability and observability of a given networked LTI system and for finding the Kalman-type decomposition.
    We also show efficient algorithms for deciding the $\sigma$-semistability of representations of acyclic quivers with self-loops if the weight $\sigma$ has the same sign for all vertices with self-loops.
    Such quiver representations and weights arise from networked LTI systems.
\end{abstract}

\section{Introduction}\label{sec:intro}
Controllability and observability of linear dynamical systems are fundamental concepts in control theory and have been widely used in real-world complex network analysis~\cite{Liu2011,Liu2016}.
For linear time-invariant (LTI) systems, controllability and observability can be characterized by the dimension of certain invariant subspaces called the controllable and unobservable subspaces, respectively (the \emph{Kalman rank condition}).
These invariant subspaces induce a canonical decomposition of LTI systems (the \emph{Kalman decomposition}), which reveals various essential properties of systems.

In many real-world applications such as social network analysis and systems biology, we often consider a \emph{networked} LTI system---a system consisting of interconnected small LTI subsystems~\cite{Hao2022}.
Networked LTI systems are also called \emph{network-of-networks}~\cite{Chapman2014} and \emph{networked MIMO systems}~\cite{Wang2016a}.
A naive approach for analyzing networked LTI systems is to simply regard them as large LTI systems and apply classical control-theoretic methods such as the Kalman decomposition.
However, this naive approach completely ignores the underlying network structure and often fails to provide useful insights into the system.
For example, the controllable and unobservable subspaces may be mixtures of substates across different subsystems. 
Therefore, these control-theoretic concepts may not be interpretable in terms of the network structure~\cite{Iudice2019}.
In fact, there seems to be no standard framework suitable for networked LTI systems with high-dimensional subsystems; 
it is even said that ``we lack a general framework to systematically explore the control of networks of networks''~\cite{Liu2016}.

This prompts us to develop new notions of controllability and observability that respect the underlying network structure.

\subsection{Our contributions}
In this paper, we study networked LTI systems through the lens of \emph{quiver representations} and we demonstrate that quiver representations provide a natural framework for this purpose.
Quiver representations are a powerful framework for studying problems in algebra and representation theory in a unified manner~\cite{Derksen2017book}.
King~\cite{King1994} introduced the notion of semistability of quiver representations with respect to a weight (\emph{$\sigma$-semistability}).
He also proved a characterization of $\sigma$-semistability in terms of the dimensions of collections of invariant subspaces (\emph{King's criterion}).
Such a collection of invariant subspaces is called a \emph{subrepresentation}.
Bader~\cite{Bader2008} pioneered the use of $\sigma$-semistability of quiver representations to study LTI systems.
He established a beautiful connection between the controllability and observability of LTI systems and the $\sigma$-semistability of the corresponding quiver representation.
In the last two decades, $\sigma$-semistability of quiver representations has found a wide range of applications in computational complexity~\cite{Gurvits2004,Garg2019}, functional analysis~\cite{Bennett2008,Garg2018a}, and statistics~\cite{Derksen2021}.
Furthermore, various efficient algorithms for deciding the $\sigma$-semistability of quiver representations have been developed~\cite{Garg2019,Ivanyos2018,Franks2018,Burgisser2018,Hamada2021,Franks2023,Iwamasa2025}.

\paragraph{Network-respecting controllability, observability, and Kalman-type decomposition.}
Inspired by Bader's result, we introduce new notions of controllability, observability, and Kalman-type decomposition that are suitable for networked LTI systems.
Our approach closely follows Bader's idea. 
We regard a given networked LTI system as a natural representation of a quiver with self-loops.
Then, we define the networked LTI system to be \emph{network-respecting controllable} if the corresponding quiver representation is $\sigma$-semistable for an appropriate weight $\sigma$.
Bader's generalized King's criterion~\cite{Bader2008} characterizes the network-respecting controllability in terms of subrepresentations, i.e., collections of invariant subspaces \emph{within} each subsystem.
These invariant subspaces in turn decompose the matrices of local and interconnection dynamics consistently with the underlying network structure.
We can also define \emph{network-respecting observability} in a natural dual fashion.
Finally, we define a new Kalman-type decomposition of networked LTI systems, which decomposes the local and interconnection dynamics matrices as well as input and output matrices, based on the network-respecting controllability and observability.
See Figure~\ref{fig:concepts} for the relationships among these concepts.

\begin{figure}[ht]
    \centering
    \tikzset{block/.style={draw,rectangle,align=center}}
    \begin{tikzpicture}[>=latex]
        \node[block] (NC) {Network-respecting \\ controllability};
        \node[block, right=6em of NC] (SS) {$\sigma$-semistability \\ for $\sigma = -\ones$};
        \node[block, below=5em of SS] (King) {$\sigma(\dimv W) \leq 0$ \\ $\sigma(\dimv W) \leq \sigma(\dimv V)$};
        \node[block, right=6em of King] (SB) {Simultaneous block \\ decomposition};
        \draw[<->] (NC) -- (SS) node[above, midway]{Def~\ref{def:nc}} node[below, midway, font=\footnotesize]{cf.~Lem~\ref{lem:ss-LTI}};
        \draw[<->] (SS) -- (King) node[right, midway, align=left]{Generalized King's criterion for \\ marked quivers (Lem~\ref{lem:king-marked})};
        \draw[<->] (King) -- (SB) node[below, midway]{Thm~\ref{thm:block-uptri-co}};
    \end{tikzpicture}
    \caption{Relations of the concepts.\label{fig:concepts}}
\end{figure}
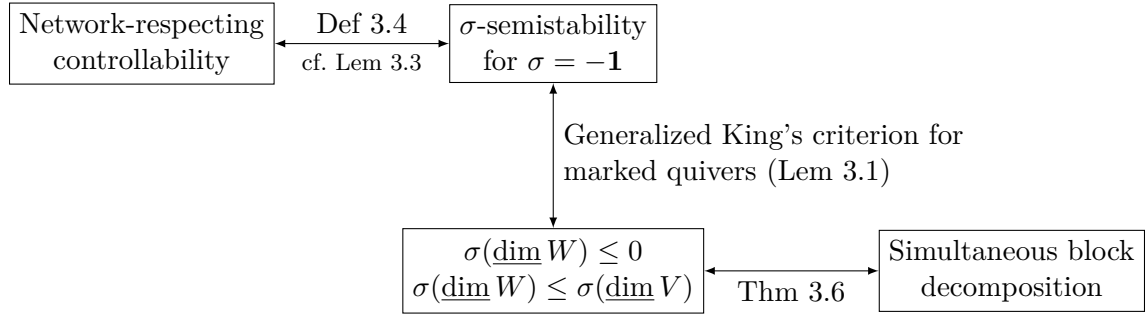

We can show that network-respecting controllability and observability are necessary conditions for the ordinary controllability and observability of networked LTI systems.
The converse is not true in general.
Therefore, our proposed notions do not replace the ordinary ones.
Rather, they identify subspaces within each subsystem that certify uncontrollability and unobservability in the ordinary sense.

\paragraph{Efficient Algorithms.}
Furthermore, we devise efficient algorithms that, given a networked LTI system, decide the network-respecting controllability and observability, and find the Kalman-type decomposition.
These algorithms use a simple subspace iteration that constructs a subrepresentation.
We show that our algorithms output subrepresentations that maximally violate the generalized King's criterion in polynomial time.

\paragraph{Extension to target control.}
Another benefit of using quiver representations is that our concepts and algorithms easily extend to the \emph{target} controllability and observability setting~\cite{Gao2014,Hao2022,Wang2024}.
The target setting is useful when we only want to control a subset of subsystems rather than the entire system, which is often the case in real-world complex networks.

\paragraph{Semistability in quivers with self-loops.}
The aforementioned results for network-respecting controllability and observability may be of significance from the computational complexity perspective, because they indicate a new class of quivers and weights for which we can decide $\sigma$-semistability and find a maximally violating subrepresentation in King's criterion in deterministic polynomial time, even under the existence of self-loops.
In contrast, the known algorithm for $\sigma$-semistability~\cite{Iwamasa2025} can deal with an arbitrary weight $\sigma$ but applies only to \emph{acyclic} quivers, i.e., quivers without directed cycles, including self-loops.
To complement our results, we prove that we can decide the $\sigma$-semistability of a quiver representation efficiently if each vertex has at most one self-loop, the quiver becomes acyclic after removing all self-loops, and the sign of the weight $\sigma$ is the same for all vertices with self-loops.
Quivers and weights appearing in network-respecting controllability and observability satisfy these conditions when the underlying network is acyclic, so this result provides another efficient algorithm for this special case.
Furthermore, this result allows $\sigma$ to take an arbitrary value for vertices without self-loops and might be useful beyond studies of networked LTI systems.
We leave further exploration to future work.

\subsection{Related work}
Semistability of quiver representations is a special case of the stability concepts in Mumford's \emph{geometric invariant theory (GIT)}.
B\"urgisser et al.~\cite{Burgisser2019} proposed an algorithmic framework for the membership problems in null-cones and moment polytopes in general GIT and devised continuous optimization algorithms for these problems.
It is known that for acyclic quivers, a representation $V$ is $\sigma$-semistable if and only if the orbit closure of $(V, 1)$ does not intersect with the origin; see, e.g., \cite[Section~9.8]{Derksen2017book}.
Therefore, $\sigma$-semistability in acyclic quivers falls into their null-cone membership framework.
However, this is not the case for cyclic quivers because $\sigma$-semistability is defined by the orbit closure of $(V, 1)$ not intersecting with the zero section rather than the origin; see Section~\ref{sec:pre} for the formal definition.
To the best of our knowledge, there is no known algorithm for deciding $\sigma$-semistability even for the special class of quivers and weights considered in this paper.

There are several lines of research on LTI systems with network structures in the literature of control theory and complex networks.
Most previous work deals with one-dimensional subsystems.
Lin~\cite{Lin1974} introduced the notion of \emph{structural controllability} and devised an efficient algorithm to decide it with bipartite matching.
Liu, Slotine, and Barab\'asi~\cite{Liu2011} utilized structural controllability to analyze real-world complex networks, which sparked a large body of follow-up research;
see, e.g., a survey by Liu and Barab\'asi~\cite{Liu2016}.

We list here several previous works that deal with high-dimensional subsystems.
Their settings can be categorized into two groups: 
all subsystems have the same local dynamics (\emph{homogeneous} subsystems) or subsystems may have different local dynamics (\emph{heterogeneous} subsystems).
Chapman et al.~\cite{Chapman2014} studied a special class of homogeneous networked LTI systems called Cartesian product systems.
Wang et al.~\cite{Wang2016a} studied homogeneous subsystems with a general network topology and a general input/output structure, assuming homogeneous interconnection dynamics.
Tang et al.~\cite{Tang2025} studied homogeneous subsystems with a multi-layer topology and heterogeneous inter-layer interconnection dynamics.
Hao et al.~\cite{Hao2022} studied target controllability in the homogeneous setting.
For heterogeneous subsystems, Zhou~\cite{Zhou2015} considered a general network topology under the assumption that every subsystem has a direct input.
Wang et al.~\cite{Wang2024} studied target controllability in multi-layered networks with each layer having homogeneous local dynamics.
These previous works focused on deriving extensions of the \emph{Popov-Belevitch-Hautus (PBH) test} for networked LTI systems and did not study Kalman-type decompositions.

Finally, we point out that Kempker et al.~\cite{Kempker2012} used a collection of invariant subspaces for control of coordinated linear systems, a special class of networked LTI systems with top-down network topology.
They proposed several controllability and observability definitions and Kalman-type decompositions for coordinated linear systems.
We remark that they did not describe their results in quiver representation terminology.
To the best of our knowledge, there is no previous work that systematically studies the Kalman-type decomposition for general networked LTI systems.

\paragraph{Organization of this paper.}
The rest of this paper is organized as follows.
Section~\ref{sec:pre} provides preliminaries on quiver representations and King's criterion.
Section~\ref{sec:control} is the main section of this paper.
After reviewing the basics of LTI systems and Bader's result, we describe our new notions of controllability, observability, and Kalman-type decomposition for networked LTI systems, as well as the algorithms.
Section~\ref{sec:reduction} presents our result on $\sigma$-semistability in acyclic quivers with self-loops.

\section{Preliminaries on quiver representations}\label{sec:pre}
We follow \cite{Derksen2017book,King1994}.
Throughout the paper, we consider an underlying field $\F = \R$ or $\F = \C$.
Let $Q = (Q_0, Q_1)$ be a quiver, where $Q_0$ is the set of vertices and $Q_1$ is the set of arcs.
We denote the head and tail vertices of an arc $a \in Q_1$ by $ha$ and $ta$, respectively.
A representation $V$ of $Q$ over $\F$ is an assignment of a finite dimensional $\F$-vector space $V(i)$ to each vertex $i \in Q_0$ and an $\F$-linear map $V(a): V(ta) \to V(ha)$ to each arc $a \in Q_1$.
By choosing a basis of each vector space $V(i)$, we can simply regard $V(a)$ as a matrix over $\F$.
The \emph{dimension vector} $\dimv V$ of $V$ is the vector $(\dimv V)(i) = \dim V(i)$ for $i \in Q_0$.
The set of all representations of $Q$ over $\F$ with a fixed dimension vector $\alpha$ is called the \emph{representation space} and denoted by $\Rep(Q, \alpha; \F)$ (or simply by $\Rep(Q, \alpha)$ when the underlying field $\F$ is clear from the context).
Let $\GL(Q, \alpha; \F) := \prod_{i \in Q_0}\GL(\alpha(i); \F)$, which is a linear algebraic group over $\F$.
Similarly, we write $\GL(Q, \alpha)$ when the underlying field $\F$ is clear from the context.
The group $\GL(Q, \alpha)$ acts on the representation space $\Rep(Q, \alpha)$ by basis change as follows:
for $g = (g_i)_{i \in Q_0} \in \GL(Q, \alpha)$ and $V \in \Rep(Q, \alpha)$, $(g \cdot V)(a) = g_{ha} V(a) g_{ta}^{-1}$ ($a \in Q_1$).
Let $\sigma: Q_0 \to \Z$ be an integer vertex weight on $Q$.
Let $\chi_\sigma(g) = \prod_{i \in Q_0} \det(g_i)^{\sigma(i)}$ be the multiplicative character of $\GL(Q, \alpha)$ associated with $\sigma$.
A representation $V$ is said to be \emph{$\sigma$-semistable}~\cite{King1994} if the orbit closure of $(V, 1) \in \Rep(Q, \alpha) \oplus \chi_\sigma$ under the $\GL(Q, \alpha)$-action does not intersect with (i.e., is disjoint from) the zero section $\Rep(Q, \alpha) \times \{0\}$.
The \emph{transposed quiver} $Q^\top$ of $Q$ is the quiver obtained by reversing the direction of all arcs in $Q$.
The \emph{transposed representation} $V^\top$ of $V$ is the representation of $Q^\top$ defined by $(V^\top)(i) = V(i)$ ($i \in Q_0$) and $(V^\top)(a) = V(a)^\top$ ($a \in Q_1$), where $V(a)^\top$ is the transpose of $V(a)$ as a matrix.
It is easy to see that $V$ is $\sigma$-semistable if and only if $V^\top$ is $(-\sigma)$-semistable.

A subrepresentation $W$ of $V$ is a set of subspaces $W(i) \leq V(i)$ ($i \in Q_0$) such that $V(a) W(ta) \leq W(ha)$ for every $a \in Q_1$.
For a dimension vector $\alpha$, we define $\sigma(\alpha) := \sum_{i \in Q_0} \sigma(i) \alpha(i)$.
King~\cite{King1994} provided a criterion for $\sigma$-semistability in terms of linear inequalities, which is known as \emph{King's criterion}.

\begin{lemma}[King's criterion \cite{King1994}]\label{lem:king}
    Let $\F = \C$ or $\F = \R$.\footnote{The original paper of King deals with the complex case ($\F = \C$) and the real case ($\F = \R$) follows from it. For completeness, we provide a proof in Appendix~\ref{sec:ss-R}.}
    A quiver representation $V$ over $\F$ is $\sigma$-semistable if and only if $\sigma(\dimv V) = 0$ and $\sigma(\dimv W) \leq 0$ for every subrepresentation $W \leq V$.
\end{lemma}

We say that $V$ is \emph{$\sigma$-stable} if it is $\sigma$-semistable and $\sigma(\dimv W) < 0$ for every nonzero proper subrepresentation $W$ of $V$.
For subrepresentations $W, W' \leq V$, we can define other subrepresentations $W + W'$ and $W \cap W'$ by $(W + W')(i) = W(i) + W'(i)$ and $(W \cap W')(i) = W(i) \cap W'(i)$ for $i \in Q_0$.
Hence, the set of all subrepresentations of $V$ forms a lattice, more precisely, a modular lattice.

A previous work~\cite{Iwamasa2025} devised a deterministic algorithm that decides the $\sigma$-semistability of a given representation $V$ of an acyclic quiver $Q$ in time polynomial in the size of $Q$, $V$, and the largest absolute value of $\sigma$.
They also provided an algorithm that finds a subrepresentation $W \leq V$ that maximizes $\sigma(\dimv W)$ with a similar time complexity.

\section{Networked LTI systems and quiver representations}\label{sec:control}
In this section, we study networked LTI systems using quiver representations.
This section is organized as follows.
In Section~\ref{subsec:pre-LTI}, we review basic concepts of LTI systems and control theory.
Section~\ref{subsec:Bader} outlines the result of Bader~\cite{Bader2008} on the connection between LTI systems and quiver representations.
In Section~\ref{subsec:networked-LTI}, we introduce our new notion of controllability and observability for networked LTI systems based on quiver representations.
In Section~\ref{subsec:Kalman-quiver}, we introduce a new Kalman-type decomposition of networked LTI systems.
Section~\ref{subsec:networked-LTI-alg} discusses algorithms for deciding the controllability and observability of networked LTI systems and computing the Kalman-type decomposition.
Finally, Section~\ref{subsec:target} discusses extensions to target controllability and observability.

\subsection{Preliminaries on LTI systems and control theory}\label{subsec:pre-LTI}
We first review basic concepts of LTI systems and control theory.
For more details, see, e.g., \cite{Antsaklis2007book}.
Consider the following LTI system:
\begin{align*}
    \dot \bx(t) &= A \bx(t) + B \bu(t),\\
    \by(t) &= C \bx(t),
\end{align*}
where $\bx(t) \in \R^n, \bu(t) \in \R^m, \by(t) \in \R^p$ are the state, input, and output vectors at time $t$, respectively, $\dot\bx(t)$ denotes the differentiation of $\bx(t)$ with respect to time $t$, and $A \in \R^{n \times n}, B \in \R^{n \times m}, C \in \R^{p \times n}$ are the system matrices. 
The LTI system is said to be \emph{controllable} if for any initial state $\bx(0)$, there exists an input $\bu(t)$ that drives the state $\bx(t)$ to the zero state in finite time.
It is well-known that the LTI system is controllable if and only if the \emph{controllability matrix} 
\begin{align}\label{eq:c-mat}
\begin{bmatrix}
    B & AB & A^2B & \cdots & A^{n-1}B
\end{bmatrix}
\end{align}
is full row rank.
Hence, the controllability of the system depends only on the pair of matrices $(A, B)$.
Sometimes we simply say that the pair $(A, B)$ is controllable.
It is easy to see that $(A, B)$ is uncontrollable if and only if there exists a nonsingular matrix $g \in \GL(n; \R)$ such that
\begin{align}\label{eq:co-block-decomp}
    gAg^{-1} = 
    \begin{bNiceMatrix}[first-row,first-col] 
            & n_1    & n_2  \\
        n_1 & A_{11} & A_{12} \\ 
        n_2 & O & A_{22} 
    \end{bNiceMatrix}, \quad
    gB = \begin{bNiceMatrix}[first-row,first-col] 
            & m  \\
        n_1 & B_1 \\ 
        n_2 & O 
    \end{bNiceMatrix},
\end{align}
where $A_{11} \in \R^{n_1 \times n_1}, A_{12} \in \R^{n_1 \times n_2}, A_{22} \in \R^{n_2 \times n_2}$, and $B_1 \in \R^{n_1 \times m}$ with $n_1 + n_2 = n$ and $n_2 > 0$.
The image of the controllability matrix is called the \emph{controllable subspace} of the system and is denoted by $\Wc$.

The dual concept of controllability is \emph{observability}.
The LTI system is said to be \emph{observable} if for any initial state $\bx(0)$, there exists a finite time $t^* > 0$ such that $\bx(0)$ can be uniquely determined from $\bu(t), \by(t)$ ($0 \leq t \leq t^*$).
It is well-known that the LTI system is observable if and only if the \emph{observability matrix} 
\begin{align}\label{eq:o-mat}
\begin{bmatrix} C \\ CA \\ CA^2 \\ \vdots \\ CA^{n-1} \end{bmatrix}
\end{align}
is full column rank.
Since the observability of the system depends only on the pair of matrices $(A, C)$, we also say that the pair $(A, C)$ is observable.
The kernel of the observability matrix is called the \emph{unobservable subspace} of the system, which we denote by $\Wno$.
By comparing the above two characterizations, we can see that $(A, C)$ is observable if and only if $(A^\top, C^\top)$ is controllable.
This is known as the \emph{duality between controllability and observability}.

The controllable and unobservable subspaces together induce the canonical decomposition of the LTI system, known as the \emph{Kalman decomposition}.
That is, there exists a nonsingular matrix $g \in \GL(n; \R)$ such that 
\begin{align*}
    gAg^{-1} &= \begin{bNiceMatrix}[first-row, first-col] 
            & \Wcno & \Wco & \Wncno & \Wnco \\
    \Wcno   & * & * & * & * \\
    \Wco    & O & * & O & * \\
    \Wncno  & O & O & * & * \\
    \Wnco   & O & O & O & *
    \end{bNiceMatrix},  &
    gB &= \begin{bNiceMatrix}[first-col] 
        \Wcno  & * \\
        \Wco   & * \\ 
        \Wncno & O \\ 
        \Wnco  & O 
    \end{bNiceMatrix}, &
    C g^{-1} &= \begin{bNiceMatrix}[first-row] 
        \Wcno & \Wco & \Wncno & \Wnco \\
        O & * & O & *
    \end{bNiceMatrix},
\end{align*}
where 
$\Wcno = \Wc \cap \Wno$, $\Wco = \Wc / (\Wc \cap \Wno)$, $\Wncno = \Wno / (\Wc \cap \Wno)$, $\Wnco = \R^n / (\Wc + \Wno)$.
Note that these four subspaces naturally arise from the following pair of chains of $A$-invariant subspaces:
\begin{center}
\begin{tikzpicture}
    \node (Zero) at (0, -2) {$\{\zeros\}$};
    \node (Wcno) at (0, -1) {$\Wc \cap \Wno$};
    \node (Wc) at (-1, 0) {$\Wc$};
    \node (Wno) at (+1, 0) {$\Wno$};
    \node (Wc+Wno) at (0, 1) {$\Wc + \Wno$};
    \node (R) at (0, 2) {$\R^n$};
    \graph{(Zero) -> (Wcno) -> {(Wc), (Wno)} -> (Wc+Wno) -> (R);};
\end{tikzpicture}
\end{center}

\subsection{LTI systems and representations of marked quivers}\label{subsec:Bader}
Bader~\cite{Bader2008} pointed out that the controllability and observability of LTI systems are closely related to the $\sigma$-semistability of quiver representations.
To describe the connection formally, we need to introduce a slightly generalized version of $\sigma$-semistability.

Let $Q = (Q_0, Q_1)$ be a quiver and $M \subseteq Q_0$ be a vertex subset.
The vertices in $M$ and $Q_0 \setminus M$ are said to be \emph{marked} and \emph{unmarked}, respectively.
Let $\sigma: M \to \Z$ be an integer weight supported on $M$. 
A representation $V$ over $\F$ of a marked quiver $Q$ with a dimension vector $\alpha$ is said to be \emph{$\sigma$-semistable} if it is $\sigma$-semistable under the action of a subgroup $\GL(M, \alpha; \F) := \prod_{i \in M}\GL(\alpha(i); \F)$ of $\GL(Q, \alpha; \F)$; see Remark~\ref{rem:marked-ss} for details.
Bader presented a generalization of King's criterion for this setting:

\begin{lemma}[generalized King's criterion \cite{Bader2008}]\label{lem:king-marked}
    Let $\F = \C$ or $\F = \R$ be the underlying field.\footnote{Bader~\cite{Bader2008} proved this result only for $\F = \C$, but his proof also works for $\F = \R$ without modification.}
    A representation $V$ over $\F$ of a marked quiver $Q$ is $\sigma$-semistable if and only if the following two conditions hold:
    \begin{enumerate}
        \item $\sigma(\dimv W) \leq 0$ for every subrepresentation $W \leq V$ with $W(i) = \{\zeros\}$ for each unmarked vertex $i \in Q_0 \setminus M$, and
        \item $\sigma(\dimv W) \leq \sigma(\dimv V)$ for every subrepresentation $W \leq V$ with $W(i) = V(i)$ for each unmarked vertex $i \in Q_0 \setminus M$.
    \end{enumerate}
\end{lemma}

\begin{remark}\label{rem:marked-ss}
    Formally, the concept of $\sigma$-semistability is adapted to a marked quiver $Q$ by referring to the subgroup $\GL(M, \alpha; \F)$ and the restricted character $\chi^M_\sigma(g) := \prod_{i \in M} \det(g_i)^{\sigma(i)}$.
    That is, a representation $V$ over $\F$ of a marked quiver $Q$ with a dimension vector $\alpha$ is said to be \emph{$\sigma$-semistable} if the orbit closure of $(V, 1) \in \Rep(Q, \alpha; \F) \oplus \chi^M_\sigma$ under the $\GL(M, \alpha; \F)$-action does not intersect with $\Rep(Q, \alpha; \F) \times \{0\}$.
    In what follows, however, we do not use this definition of $\sigma$-semistability for marked quivers;
    we only need the generalized King's criterion (Lemma~\ref{lem:king-marked}).
\end{remark}

For $M = Q_0$, Lemma~\ref{lem:king-marked} reduces to the original King's criterion (Lemma~\ref{lem:king}).
We emphasize that Lemma~\ref{lem:king-marked} even applies to marked quivers with cycles.
This generality plays a substantial role in Section~\ref{subsec:networked-LTI}.

We say that a representation $V$ of a marked quiver $Q$ is \emph{$\sigma$-stable} if
(i) it is $\sigma$-semistable; 
(ii) $\sigma(\dimv W) < 0$ for every nonzero proper subrepresentation $W$ of $V$ with $W(i) = \{\zeros\}$ for each unmarked vertex $i \in Q_0 \setminus M$;
and (iii) $\sigma(\dimv W) < \sigma(\dimv V)$ for every nonzero proper subrepresentation $W$ of $V$ with $W(i) = V(i)$ for each unmarked vertex $i \in Q_0 \setminus M$.

\begin{figure}
    \centering
    \includegraphics[width=0.2\textwidth]{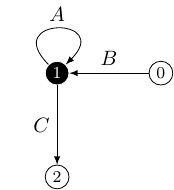}
    \caption{The quiver for LTI systems.\label{fig:LTI-quiver}}
\end{figure}

Consider the marked quiver $Q$ in Figure~\ref{fig:LTI-quiver} with $Q_0 = \{0, 1, 2\}$ and $M = \{1\}$.
An LTI system $(A, B, C)$ can be seen as a real representation $V$ of $Q$.
A weight $\sigma$ on the single marked node $1$ can be regarded as an integer.
The controllability and observability of the LTI system can be characterized in terms of the $\sigma$-semistability of $V$ as follows.

\begin{lemma}[\cite{Bader2008}]\label{lem:ss-LTI}
    \quad
    \begin{enumerate}
        \item When $\sigma \leq -1$, $(A, B)$ is controllable if and only if $V$ is $\sigma$-semistable;
        \item When $\sigma \geq 1$, $(A, C)$ is observable if and only if $V$ is $\sigma$-semistable; and
        \item When $\sigma = 0$, an LTI system $(A, B, C)$ is controllable and observable if and only if $V$ is $\sigma$-stable.
    \end{enumerate}
\end{lemma}

To gain intuition for the lemma, let us first focus on the controllability of $(A, B)$.
Since $\sigma \leq -1$, the first condition in the generalized King's criterion (Lemma~\ref{lem:king-marked}) is met trivially.
The inequality of the second condition $\sigma(\dimv W) \leq \sigma(\dimv V)$ is equivalent to $\dim W(1) \geq n$, where $1$ is the unique marked vertex in $Q$.
Therefore, it suffices to check a subrepresentation $W$ with $W(0) = \R^m$, $W(2) = \R^p$, and $\dim W(1)$ minimum.
It is easy to see that such $W$ yields
\begin{align*}
    W(1) = \sum_{k=0}^\infty A^k \im B = \sum_{k=0}^{n-1} A^k \im B.
\end{align*} 
Note that $W(1)$ is the minimum $A$-invariant subspace containing $\im B$, which equals the controllable subspace.
Thus, the $\sigma$-semistability of $V$ is equivalent to the controllability matrix \eqref{eq:c-mat} being full row rank.

Similarly, consider the observability of $(A, C)$.
Since $\sigma \geq 1$, the second condition in the generalized King's criterion is met trivially.
The inequality of the first condition reads $\dim W(1) \leq 0$.
Therefore, it suffices to consider a subrepresentation $W$ with $W(0) = \{\zeros\}$, $W(2) = \{\zeros\}$, and $\dim W(1)$ maximum.
Such a subspace $W(1)$ is the maximum $A$-invariant subspace contained in $\ker C$, which equals the unobservable subspace.
Thus, the $\sigma$-semistability of $V$ is equivalent to the observability matrix \eqref{eq:o-mat} being full column rank.

Another interesting aspect of the lemma is that the duality between controllability and observability is naturally encoded as the duality between the $\sigma$-semistability of $V$ and the $(-\sigma)$-semistability of $V^\top$~\cite{Stupariu2014}.

In the next section, we will extend the connection between $\sigma$-semistability and controllability/observability to networked LTI systems.

\subsection{Networked LTI systems}\label{subsec:networked-LTI}
In many real-world applications, we often consider a \emph{networked} LTI system, a system consisting of interconnected small LTI subsystems.
Networked LTI systems are also called \emph{network-of-networks} or \emph{networked MIMO systems} in the literature~\cite{Chapman2014,Wang2016a}.

Formally, a networked LTI system is defined as follows.
Let $N = (N_0, N_1)$ be a quiver, where $N_0$ is the set of subsystems and $N_1$ is the set of interconnecting arcs between subsystems.
We assume that $N$ has no self-loops.
Let $\alpha \in \Z^{N_0}$ be a dimension vector representing the state dimensions of the subsystems.
Each subsystem $i \in N_0$ may have a direct input of dimension $m(i)$ and a direct output of dimension $p(i)$.
We remark that not all subsystems need to have direct inputs or outputs.
Let $\Nin_0$ and $\Nout_0$ be the sets of subsystems with direct inputs and outputs, respectively.
Note that $\Nin_0 \cap \Nout_0$ may be nonempty.
We have the following four types of system matrices:
\begin{itemize}
    \item \emph{[local dynamics]} $A(i) \in \R^{\alpha(i) \times \alpha(i)}$ for each subsystem $i \in N_0$,
    \item \emph{[local input]} $B(i) \in \R^{\alpha(i) \times m(i)}$ for each subsystem $i \in \Nin_0$ with a direct input of dimension $m(i)$,
    \item \emph{[local output]} $C(i) \in \R^{p(i) \times \alpha(i)}$ for each subsystem $i \in \Nout_0$ with a direct output of dimension $p(i)$, and
    \item \emph{[interconnection dynamics]} $V(a) \in \R^{\alpha(ha) \times \alpha(ta)}$ for each interconnecting arc $a \in N_1$.
\end{itemize}
The state-space representation of the networked LTI system is as follows:
\begin{equation}\label{eq:networked-LTI}
\begin{aligned}
    \dot \bx_i &= 
    A(i) \bx_i + \sum_{a \in N_1: ha = i} V(a) \bx_{ta} + B(i) \bu_i 
    & & (i \in \Nin_0), \\
    \dot \bx_i &= 
    A(i) \bx_i + \sum_{a \in N_1: ha = i} V(a) \bx_{ta} 
    & & (i \in N_0 \setminus \Nin_0),  \\
    \by_i      &= 
    C(i) \bx_i 
    & & (i \in \Nout_0),
\end{aligned}
\end{equation}
where $\bx_i = \bx_i(t)$, $\bu_i = \bu_i(t)$, and $\by_i = \by_i(t)$ are the state, input, and output vectors of subsystem $i$, respectively.

Of course, a networked LTI system is just a (very high-dimensional) LTI system.
Therefore, we can define the controllability, observability, and Kalman decomposition of a networked LTI system in the same way as for ordinary LTI systems.
However, such naive definitions ignore the underlying network structure.
For example, the controllable and unobservable subspaces may be mixtures of several state vectors from different subsystems.
Consequently, the Kalman decomposition does not necessarily respect the network structure.
Therefore, these control-theoretic concepts are often difficult to interpret.
We can overcome this issue by formulating new notions of controllability, observability, and Kalman decomposition for networked LTI systems, based on $\sigma$-semistability of quiver representations.

To be specific, we will construct another marked quiver $Q$ from $N$ as follows.
First, we start with $N$ and mark each vertex $i \in N_0$.
Second, we place a single self-loop at each subsystem $i \in N_0$.
Third, for each subsystem $i \in \Nin_0$, we create a new unmarked node $s_i$ and add an arc from $s_i$ to $i$.
Finally, for each subsystem $i \in \Nout_0$, we create a new unmarked vertex $o_i$ and add an arc from $i$ to $o_i$.
Let $Q$ be the resulting quiver.
Note that 
\[
    Q_0 = N_0 \cup \{ s_i \mid i \in \Nin_0 \} \cup \{ o_i \mid i \in \Nout_0 \}
\]
and that the set of marked vertices of $Q$ is exactly $N_0$.
See Figure~\ref{fig:networked-LTI-quiver} for an example of the construction.

A representation of $Q$ consists of the following four types of matrices:
\begin{itemize}
    \item $A(i)$ for the self-loop at each subsystem $i \in N_0$,
    \item $B(i)$ for each subsystem $i \in \Nin_0$,
    \item $C(i)$ for each subsystem $i \in \Nout_0$, and
    \item $V(a)$ for each interconnecting arc $a \in N_1$.
\end{itemize}
Thus, there is a natural one-to-one correspondence between a networked LTI system with the underlying network topology $N$ and a representation $V$ of the marked quiver $Q$.

\begin{figure}
    \centering
    \subcaptionbox{\label{subfig:networked-LTI}}{\includegraphics[width=0.3\textwidth]{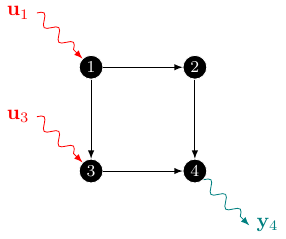}}
    \hspace{3em}
    \subcaptionbox{\label{subfig:quiver-repr}}{\includegraphics[width=0.3\textwidth]{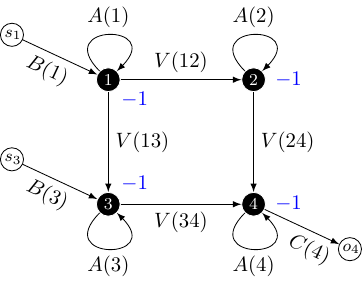}}
    \caption{(\subref{subfig:networked-LTI}) A networked LTI system. The underlying network quiver $N$ is shown in the black nodes and arcs. The wave arcs indicate direct inputs to subsystems 1 and 3 (in \textcolor{red}{red}) and a direct output from subsystem 4 (in \textcolor{teal}{green}), respectively. 
    (\subref{subfig:quiver-repr}) The corresponding marked quiver $Q$, the representation $V$, and the weight $\sigma$ on $N_0$. The values of $\sigma(i) = -1$ are attached to each vertex in $N_0$ (in \textcolor{blue}{blue}). \label{fig:networked-LTI-quiver}}
\end{figure}

We now define our new notion of network-respecting controllability for networked LTI systems.

\begin{definition}[Network-respecting controllability of networked LTI systems]\label{def:nc}
We say that a networked LTI system \eqref{eq:networked-LTI} is \emph{network-respecting controllable} if the corresponding representation $V$ of the marked quiver $Q$ is $\sigma$-semistable, where $\sigma$ is a weight defined by $\sigma(i) = -1$ ($i \in N_0$) and $\sigma(i) = 0$ ($i \in Q_0 \setminus N_0$).
\end{definition}

By Lemma~\ref{lem:ss-LTI}, the standard controllability of the networked LTI system is a special case of the network-respecting controllability in which $Q$ is the quiver in Figure~\ref{fig:LTI-quiver}.

Let us use the generalized King's criterion (Lemma~\ref{lem:king-marked}) to check the network-respecting controllability of a given networked LTI system.
Since $\sigma$ is nonpositive, the first condition of the generalized King's criterion is trivially satisfied.
Therefore, a networked LTI system is network-respecting controllable if and only if $\sigma(\dimv W) \leq \sigma(\dimv V)$ for every subrepresentation $W \leq V$ with $W(i) = V(i)$ for each unmarked $i \in Q_0 \setminus N_0$.
Let $\caL_c$ be the lattice of all subrepresentations $W \leq V$ with $W(i) = V(i)$ ($i \in Q_0 \setminus N_0$).
By the modular lattice structure of subrepresentations, there exists a unique minimum subrepresentation in $\caL_c$.
We call this minimum subrepresentation the \emph{controllable subrepresentation} and denote it by $\Wc$.
Then, the network-respecting controllability is equivalent to $\sigma(\dimv \Wc) \leq \sigma(\dimv V)$ since $\sigma$ is nonpositive.
This inequality reads
\[
    -\sum_{i \in N_0} \dim \Wc(i) \leq -\sum_{i \in N_0} \dim V(i),
\]
which is equivalent to $\dim \Wc(i) = \dim V(i)$ for every $i \in N_0$.
In the case of the quiver in Figure~\ref{fig:LTI-quiver}, the subspace $\Wc(1)$ of the controllable subrepresentation $\Wc$ is exactly the controllable subspace of the LTI system.
We record this observation in the following lemma.

\begin{lemma}\label{lem:nc-Wc}
    A networked LTI system \eqref{eq:networked-LTI} is network-respecting controllable if and only if the controllable subrepresentation $\Wc$ satisfies $\Wc(i) = V(i)$ for every subsystem $i \in N_0$.
\end{lemma}

The following theorem gives a characterization of network-respecting controllability of networked LTI systems in terms of block decompositions similar to \eqref{eq:co-block-decomp}.

\begin{theorem}\label{thm:block-uptri-co}
    A networked LTI system \eqref{eq:networked-LTI} is not network-respecting controllable if and only if there exist $g \in \GL(N, \alpha; \R)$ and a dimension vector $\beta \leq \alpha$ such that 
    \begin{alignat}{2}
        g_i A(i) g_i^{-1} &=  
        \begin{bNiceMatrix}[first-row,first-col] 
                     & \beta(i) &   \\
            \beta(i) & *   & * \\ 
                     & O   & *
        \end{bNiceMatrix} &\quad& (i \in N_0),
        \label{eq:block-uptri-co-A}
        \\
        g_{ha} V(a) g_{ta}^{-1} &=  
        \begin{bNiceMatrix}[first-row,first-col] 
                      & \beta(ta) &   \\
            \beta(ha) & *   & * \\ 
                      & O   & *
        \end{bNiceMatrix} &\quad& (a \in N_1),
        \label{eq:block-uptri-co-V}
        \\
        g_i B(i) &= 
        \begin{bNiceMatrix}[first-row,first-col] 
                     & m(i) \\
            \beta(i) & *   \\ 
                     & O   
        \end{bNiceMatrix} &\quad& (i \in \Nin_0),
        \label{eq:block-uptri-co-B}
    \end{alignat}
    and $\beta(i) < \alpha(i)$ for some subsystem $i \in N_0$ (with or without a direct input).
\end{theorem}
\begin{proof}
    Suppose that the networked LTI system is not network-respecting controllable.
    Let $\beta$ be the dimension vector of the controllable subrepresentation $\Wc$.
    By choosing $g \in \GL(N, \alpha; \R)$ so that $\Wc(i)$ is the coordinate subspace of the first $\beta(i)$ coordinates of $g_i A(i) g_i^{-1}$ for each $i \in N_0$, we obtain the desired block decomposition.
    Furthermore, Lemma~\ref{lem:nc-Wc} implies that $\beta(i) < \alpha(i)$ for some subsystem $i \in N_0$.

    Conversely, if the above block decomposition exists, then let $W(i)$ be the coordinate subspace of the first $\beta(i)$ coordinates for each $i \in N_0$.
    Then, $W$ forms a subrepresentation in $\caL_c$.
    Since $\dim W(i) = \beta(i) < \alpha(i) = \dim V(i)$ for some $i \in N_0$, we have $\sigma(\dimv W) > \sigma(\dimv V)$, and therefore the networked LTI system is not network-respecting controllable.
\end{proof}

\begin{remark}\label{rem:nc-del}
    As one can see from the theorem, the network-respecting controllability of a networked LTI system \eqref{eq:networked-LTI} depends neither on the set $\Nout_0$ nor on the output matrices $C(i)$ ($i \in \Nout_0$).
    This corresponds to the fact that the ordinary controllability of LTI systems does not depend on the output matrix $C$.
    Therefore, one could alternatively formulate network-respecting controllability using a smaller quiver that is obtained by deleting the unmarked vertices $o_i$ ($i \in \Nout_0$) and the adjacent arcs from $Q$.
    In this paper, however, we use the quiver $Q$ without the deletion because it highlights the duality between network-respecting controllability and observability in a unified manner, and because Bader~\cite{Bader2008} used the quiver without the deletion in his theorem (see Figure~\ref{fig:LTI-quiver} and Lemma~\ref{lem:ss-LTI}).
\end{remark}

\begin{corollary}
    If a networked LTI system \eqref{eq:networked-LTI} is not network-respecting controllable, then it is not controllable as an LTI system.
\end{corollary}
\begin{proof}
    Let $n = \sum_{i \in N_0} \alpha(i)$ and $m = \sum_{i \in \Nin_0} m(i)$ and let $\tilde A, \tilde B$ be the aggregated system matrices of the networked LTI system as a single LTI system.
    Namely, $\tilde A$ is an $n \times n$ block matrix whose diagonal blocks are $A(i)$ and whose off-diagonal blocks are either $V(a)$ or the zero matrix.
    Similarly, $\tilde B$ is an $n \times m$ block matrix whose blocks are either $B(i)$ or the zero matrix.
    Take the controllable subrepresentation $W_c$ and let $U = \bigoplus_{i \in N_0} W_c(i)$ be the corresponding subspace of $\R^n$. 
    Then, $U$ is a $\tilde A$-invariant subspace containing $\im \tilde B$.
    By Lemma~\ref{lem:nc-Wc}, if the networked LTI system is not network-respecting controllable, then $\dim U = \sum_{i \in N_0} \dim \Wc(i) < \sum_{i \in N_0} \alpha(i) = n$.
    Since the controllable subspace is the minimum $\tilde A$-invariant subspace containing $\im \tilde B$, it is contained in $U$ and therefore has dimension strictly less than $n$.
\end{proof}

\begin{remark}
    Here, we compare our notion of network-respecting controllability with other related notions.
    First, there exists a networked LTI system that is network-respecting controllable but not controllable as an LTI system. 
    Consider the networked LTI system in Figure~\ref{fig:networked-LTI-nc} with $\alpha(i) = 1$ for all subsystems and $m(1) = 1$.
    Suppose that the local dynamics matrix satisfies $A(i) = 0$ for each subsystem $i$.
    The interconnection dynamics are given by $a_1, a_2 \in \R$ and the input matrix at subsystem $1$ is simply a scalar $b \in \R$.
    We assume that $a_1, a_2, b \neq 0$.
    Then, the controllable subrepresentation $\Wc$ is given by $\Wc(1) = \R$, $\Wc(2) = \R$, and $\Wc(3) = \R$.
    Therefore, $\sigma(\dimv \Wc) = -3 = \sigma(\dimv V)$, and the system is network-respecting controllable.
    On the other hand, the aggregated system matrices are given by
    \begin{alignat*}{2}
        \tilde A &= 
        \begin{bmatrix}
            0 & 0 & 0 \\
            a_1 & 0 & 0 \\
            a_2 & 0 & 0
        \end{bmatrix},
        & \quad
        \tilde B &= \begin{bmatrix}
            b \\
            0 \\
            0
        \end{bmatrix}.
    \end{alignat*}
    Then, the controllability matrix \eqref{eq:c-mat} is given by
    \begin{align*}
        \begin{bmatrix}
            b & 0     & 0 \\
            0 & b a_1 & 0 \\
            0 & b a_2 & 0 
        \end{bmatrix},
    \end{align*}
    which is not full row rank.
    Therefore, the system is not controllable as an LTI system.

    Second, if a networked LTI system is network-respecting controllable, then every subsystem is reachable\footnote{Here, the reachability is graph-theoretic and is different from the reachability of states in control theory.} by a directed path from some subsystem with a direct input.
    To see this, let $X \subseteq N_0$ be the set of subsystems reachable from some subsystem with a direct input.
    Then, subspaces $W(i) = V(i)$ for $i \in X$ and $W(i) = \{\zeros\}$ for $i \in N_0 \setminus X$ form a subrepresentation of $V$.
    If there exists an unreachable subsystem $i \in N_0 \setminus X$, then one can take $\beta(i) = 0$ and therefore the networked LTI system is not network-respecting controllable.
\end{remark}

\begin{figure}
    \centering
    \includegraphics[width=3cm]{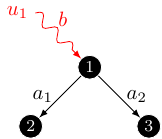}
    \caption{A networked LTI system that is network-respecting controllable but not controllable as an LTI system.}\label{fig:networked-LTI-nc}
\end{figure}

We now define the notion of observability for networked LTI systems in a dual fashion.

\begin{definition}[Network-respecting observability of networked LTI systems]
We say that a networked LTI system \eqref{eq:networked-LTI} is \emph{network-respecting observable} if the corresponding representation $V$ of the marked quiver $Q$ is $\sigma$-semistable, where $\sigma$ is a weight defined by $\sigma(i) = 1$ ($i \in N_0$) and $\sigma(i) = 0$ ($i \in Q_0 \setminus N_0$).
\end{definition}

As with network-respecting controllability, the generalized King's criterion (Lemma~\ref{lem:king-marked}) implies that a networked LTI system is network-respecting observable if and only if $\sigma(\dimv W) \leq 0$ for every subrepresentation $W \leq V$ with $W(i) = \{\zeros\}$ for each unmarked $i \in Q_0 \setminus N_0$.
Note that the second condition in the generalized King's criterion is trivially satisfied since $\sigma$ is nonnegative.
Let $\caL_o$ be the lattice of all subrepresentations $W \leq V$ with $W(i) = \{\zeros\}$ ($i \in Q_0 \setminus N_0$).
We call the unique \emph{maximum} subrepresentation in $\caL_o$ the \emph{unobservable subrepresentation} and denote it by $\Wno$.
Then, the network-respecting observability is equivalent to $\sigma(\dimv \Wno) \leq 0$, which is further equivalent to $\Wno(i) = \{\zeros\}$ for every $i \in N_0$.
Again, in the case of the quiver in Figure~\ref{fig:LTI-quiver}, the subspace $\Wno(1)$ is exactly the unobservable subspace of the LTI system.
Thus, we have the following network-respecting observability version of Lemma~\ref{lem:nc-Wc}.

\begin{lemma}\label{lem:no-Wno}
    A networked LTI system \eqref{eq:networked-LTI} is network-respecting observable if and only if the unobservable subrepresentation $\Wno$ satisfies $\Wno(i) = \{\zeros\}$ for every subsystem $i \in N_0$.
\end{lemma}

The next theorem follows from an argument similar to the proof of Theorem~\ref{thm:block-uptri-co}.

\begin{theorem}\label{thm:block-uptri-obs}
    A networked LTI system \eqref{eq:networked-LTI} is not network-respecting observable if and only if there exist $g \in \GL(N, \alpha; \R)$ and a dimension vector $\gamma \leq \alpha$ such that 
    \begin{alignat}{2}
        g_i A(i) g_i^{-1} &=  
        \begin{bNiceMatrix}[first-row,first-col] 
                     & \gamma(i) &   \\
            \gamma(i) & *   & * \\ 
                     & O   & *
        \end{bNiceMatrix} &\quad& (i \in N_0),
        \label{eq:block-uptri-obs-A}
        \\
        g_{ha} V(a) g_{ta}^{-1} &=  
        \begin{bNiceMatrix}[first-row,first-col] 
                      & \gamma(ta) &   \\
            \gamma(ha) & *   & * \\ 
                      & O   & *
        \end{bNiceMatrix} &\quad& (a \in N_1),
        \label{eq:block-uptri-obs-V}
        \\
        C(i) g_i^{-1} &= 
        \begin{bNiceMatrix}[first-row,first-col] 
                 & \gamma(i) &   \\
            p(i) & O & * \\
        \end{bNiceMatrix} &\quad& (i \in \Nout_0),
        \label{eq:block-uptri-obs-C}
    \end{alignat}
    and $\gamma(i) > 0$ for some subsystem $i \in N_0$ (with or without a direct output).
\end{theorem}

\begin{corollary}
    If a networked LTI system \eqref{eq:networked-LTI} is not network-respecting observable, then it is not observable as an LTI system.
\end{corollary}

\begin{remark}
    Similar to Remark~\ref{rem:nc-del}, the network-respecting observability of a networked LTI system~\eqref{eq:networked-LTI} depends neither on the set $\Nin_0$ nor on the input matrices $B(i)$ ($i \in \Nin_0$).
\end{remark}

\subsection{Kalman-type decomposition for networked LTI systems}\label{subsec:Kalman-quiver}
We now introduce our Kalman-type decomposition for networked LTI systems based on network-respecting controllability and observability.

\begin{theorem}\label{thm:Kalman-decomp-NTI}
    Let $\Wc$ and $\Wno$ be the controllable and unobservable subrepresentations, respectively.
    Then, there exists $g \in \GL(N, \alpha; \R)$ such that the following simultaneous block decomposition holds:    
\begin{align*}
    g_iA(i)g_i^{-1} &= \begin{bNiceMatrix}[first-row, first-col] 
            & \Wcno(i) & \Wco(i) & \Wncno(i) & \Wnco(i) \\
    \Wcno  (i) & * & * & * & * \\
    \Wco   (i) & O & * & O & * \\
    \Wncno (i) & O & O & * & * \\
    \Wnco  (i) & O & O & O & *
    \end{bNiceMatrix}
     &\quad& (i \in N_0), \\
    g_{ha}V(a)g_{ta}^{-1} &= \begin{bNiceMatrix}[first-row, first-col]
               & \Wcno(ta) & \Wco(ta) & \Wncno(ta) & \Wnco(ta) \\
    \Wcno  (ha) & * & * & * & * \\
    \Wco   (ha) & O & * & O & * \\
    \Wncno (ha) & O & O & * & * \\
    \Wnco  (ha) & O & O & O & *
    \end{bNiceMatrix}
    &\quad& (a \in N_1), \\
    g_iB(i) &= \begin{bNiceMatrix}[first-col] 
        \Wcno  (i) & * \\ 
        \Wco   (i) & * \\ 
        \Wncno (i) & O \\ 
        \Wnco  (i) & O
    \end{bNiceMatrix}
    &\quad& (i \in \Nin_0), \\
    C(i) g_i^{-1} &= \begin{bNiceMatrix}[first-row] 
            \Wcno(i) & \Wco(i) & \Wncno(i) & \Wnco(i) \\
            O & * & O & *
    \end{bNiceMatrix} 
    &\quad& (i \in \Nout_0),
\end{align*}
where we define 
$\Wcno(i) = \Wc(i) \cap \Wno(i)$, $\Wco(i) = \Wc(i) / (\Wc(i) \cap \Wno(i))$, $\Wncno(i) = \Wno(i) / (\Wc(i) \cap \Wno(i))$, and $\Wnco(i) = V(i) / (\Wc(i) + \Wno(i))$ for each $i \in N_0$.
\end{theorem}

\begin{proof}
    Since $\Wc$, $\Wno$, and $\Wcno = \Wc \cap \Wno$ are subrepresentations, we have the block structure of $A(i)$ and $V(a)$.
    Since $\Wc(s_i) = \R^{m(i)}$ for each $i \in \Nin_0$, we have $\im B(i) = B(i) \Wc(s_i) \leq \Wc(i)$, which implies the desired block structure of $B(i)$.
    Similarly, for each $i \in \Nout_0$, we have $\Wno(i) \leq \ker C(i)$ by $\Wno(o_i) = \{\zeros\}$.
    Thus, we have the desired block structure of $C(i)$.
\end{proof}

Note that the above Kalman-type decomposition corresponds to the following pair of chains of subrepresentations:
\begin{center}
\begin{tikzpicture}
    \node (Zero) at (0, -2) {$\{\zeros\}$};
    \node (Wcno) at (0, -1) {$\Wc \cap \Wno$};
    \node (Wc) at (-1, 0) {$\Wc$};
    \node (Wno) at (+1, 0) {$\Wno$};
    \node (Wc+Wno) at (0, 1) {$\Wc + \Wno$};
    \node (V) at (0, 2) {$V$};
    \graph{(Zero) -> (Wcno) -> {(Wc), (Wno)} -> (Wc+Wno) -> (V);};
\end{tikzpicture}
\end{center}
Here, $\{\zeros\}$ denotes the zero subrepresentation of $V$.
Also, the above Kalman-type decomposition recovers the standard Kalman decomposition when $Q$ is the quiver in Figure~\ref{fig:LTI-quiver}.

\subsection{Algorithms for controllability, observability, and Kalman-type decomposition for networked LTI systems}\label{subsec:networked-LTI-alg}
In this section, we describe algorithms for network-respecting controllability, observability, and the Kalman-type decomposition for networked LTI systems.

We show our algorithm for network-respecting controllability in Algorithm~\ref{alg:nc}.
Our algorithm is a simple subspace iteration.
We maintain a subspace $W^{(k)}(i)$ for each $i \in N_0$ and each $k = 0, 1, 2, \dots$.
Initially, we set 
\begin{align}\label{eq:W-iter-init}
W^{(0)}(i) := 
\begin{cases}
\im B(i) & (i \in \Nin_0), \\
\{\zeros\} & (i \in N_0 \setminus \Nin_0).
\end{cases}
\end{align}
Then, we iteratively define a new subspace $W^{(k+1)}(i)$ by 
\begin{align}\label{eq:W-iter}
    W^{(k+1)}(i) := W^{(k)}(i) + A(i) W^{(k)}(i) + \sum_{a \in N_1: ha = i} V(a) W^{(k)}(ta)
\end{align}
for each $i \in N_0$.
This subspace iteration generates a flag $W^{(0)}(i) \leq W^{(1)}(i) \leq W^{(2)}(i) \leq \cdots$ in $V(i)$ for each $i \in N_0$.
If $W^{(k+1)}(i) = W^{(k)}(i)$ for all $i \in N_0$, then
\[
    W^{(k)}(i) = W^{(k)}(i) + A(i) W^{(k)}(i) + \sum_{a \in N_1: ha = i} V(a) W^{(k)}(ta),
\]
i.e., $W^{(k)}(i)$ is $A(i)$-invariant for each $i \in N_0$ and $V(a)W^{(k)}(ta) \leq W^{(k)}(ha)$ for each $a \in N_1$.
By setting $W^{(k)}(i) := V(i)$ for each unmarked vertex $i \in Q_0 \setminus N_0$, we can extend these subspaces to a subrepresentation $W^{(k)}$ of $V$.
This is the output of the algorithm.
We will show that this algorithm outputs the controllable subrepresentation $\Wc$.

\begin{algorithm}
    \caption{Algorithm for computing the controllable subrepresentation $\Wc$}\label{alg:nc}
\begin{algorithmic}[1]
    \STATE Initialize $W^{(0)}(i)$ by \eqref{eq:W-iter-init} for each $i \in N_0$.
    \FOR{$k = 0, 1, 2, \dots$}
        \STATE Define $W^{(k+1)}(i)$ for each $i \in N_0$ by \eqref{eq:W-iter}.
        \IF{$W^{(k+1)}(i) = W^{(k)}(i)$ for all $i \in N_0$}
            \RETURN $(W^{(k)}(i) \mid i \in N_0) \cup (V(i) \mid i \in Q_0 \setminus N_0)$.
        \ENDIF
    \ENDFOR
\end{algorithmic}
\end{algorithm}

\begin{remark}
    As seen in Remark~\ref{rem:nc-del}, the network-respecting controllability of a networked LTI system depends neither on the set $\Nout_0$ nor on the output matrices $C(i)$ ($i \in \Nout_0$).
    Correspondingly, Algorithm~\ref{alg:nc} does not use the output matrices $C(i)$ ($i \in \Nout_0$).
    Note that the algorithm uses $N$ rather than the extended quiver $Q$.
\end{remark}

We denote $\sum_{i \in N_0}\alpha(i)$ by $|\alpha|$.
We show the correctness and time complexity of Algorithm~\ref{alg:nc} in the following lemma.

\begin{lemma}\label{lem:alg}
   Algorithm~\ref{alg:nc} outputs the controllable subrepresentation $\Wc$ in time 
   \begin{align}\label{eq:Tin}
   O\left(\sum_{i \in \Nin_0}m(i)\alpha(i)\min\{m(i), \alpha(i)\} + |\alpha| \bigg(\sum_{i \in N_0} \alpha(i)^3 + \sum_{a \in N_1} \alpha(ha)\alpha(ta)(\alpha(ha) + \alpha(ta))\bigg)\right).
   \end{align}
   Furthermore, if all system matrices are rational, then the bit complexity of Algorithm~\ref{alg:nc} is polynomial in the input size.
\end{lemma}
\begin{proof}
    First, we claim that $W^{(k)}(i) \leq \Wc(i)$ for each $i \in N_0$ and $k = 0, 1, 2, \dots$.
    We show this claim by induction on $k$.
    For $k=0$, we have $W^{(0)}(i) = \im B(i) \leq \Wc(i)$ for each $i \in \Nin_0$ and $W^{(0)}(i) = \{\zeros\} \leq \Wc(i)$ for each $i \in N_0 \setminus \Nin_0$.
    Thus the claim holds for $k=0$.
    Next, suppose that $W^{(k)}(i) \leq \Wc(i)$ for each $i \in N_0$.
    Since $\Wc(i)$ is $A(i)$-invariant for each $i \in N_0$ and $V(a)\Wc(ta) \leq \Wc(ha)$ for each $a \in N_1$, we have $W^{(k+1)}(i) \leq \Wc(i)$ by \eqref{eq:W-iter}, which completes the proof of the induction step and hence the claim.

    Let $W$ be the subrepresentation output by the algorithm.
    By construction, $W \in \caL_c$.
    Also, $W \leq \Wc$ by the claim.
    Since $\Wc$ is the minimal subrepresentation of $\caL_c$, we must have $W = \Wc$, which shows the correctness of the algorithm.

    Let us analyze the time complexity.
    Whenever some subspace is updated, the sum of the dimensions $\sum_{i \in N_0} \dim W^{(k)}(i)$ increases by at least one.
    Obviously, $\sum_{i \in N_0} \dim W^{(k)}(i) \leq |\alpha|$, so there are at most $|\alpha|$ iterations.

    Next, consider the time complexity of each iteration.
    For each $i \in N_0$, let $P^{(k)}(i)$ be a matrix whose columns form a basis of $W^{(k)}(i)$.
    Initial matrix $P^{(0)}(i)$ can be computed in $O(m(i)\alpha(i)\min\{m(i), \alpha(i)\})$ time by Gaussian elimination for each $i \in \Nin_0$.
    Given $P^{(k)}(i)$, we can construct the next matrix $P^{(k+1)}(i)$ as follows:
    Find a maximum-size linearly independent subset of the columns of matrices
    \begin{align}\label{eq:GE-mat}
        P^{(k)}(i), A(i) P^{(k)}(i), \text{ and } V(a) P^{(k)}(ta) \; (a \in N_1 \text{ with } ha = i),
    \end{align}
    then set the columns of $P^{(k+1)}(i)$ to be this linearly independent subset.
    In \eqref{eq:GE-mat}, the second matrix can be computed in $O(\alpha(i)^3)$ time and the third matrices can be computed in $O(\sum_{a \in N_1: ha = i} \alpha(i) \alpha(ta)^2)$ time.
    Such a set of columns can be found by Gaussian elimination to these matrices in $O(\alpha(i)^2 (\alpha(i) + \sum_{a \in N_1: ha = i} \alpha(ta)))$ time.
    Therefore, the total time complexity to compute $P^{(k+1)}(i)$ for all $i \in N_0$ is 
    \[
        O\bigg(\sum_{i \in N_0} \alpha(i)^3 + \sum_{a \in N_1} (\alpha(ha)^2\alpha(ta) + \alpha(ha) \alpha(ta)^2) \bigg)
    \]
    time. 
    This proves the desired time complexity.
    Note that Gaussian elimination is used only for finding the column subset and the actual columns of $P^{(k+1)}(i)$ are taken from the matrices in \eqref{eq:GE-mat}, which is crucial for the following bit complexity analysis.

    Finally, we show that the bit complexity of this algorithm is polynomial.
    Inductively, we can show the columns of $P^{(k)}(i)$ are taken from the columns of matrix products in the form of 
    \begin{align*}
        V(a_\ell) \cdots V(a_1) P^{(0)}(ta_1),
    \end{align*}
    for a directed path\footnote{Note that a path here can have repeated vertices and arcs.} $a_1 \cdots a_\ell$ in $Q$ from $ta_1 \in \Nin_0$ to $i$ of length $\ell \leq k$.
    Since $k \leq |\alpha|$, the bit complexity of these matrix products is polynomial.
    Since the bit complexity of Gaussian elimination is polynomial in that of the input matrices~\cite[Theorem~3.3]{Schrijver1986book}, the bit complexity of the algorithm is polynomial.
\end{proof}

One can also find the unobservable subrepresentation $\Wno$ by applying Algorithm~\ref{alg:nc} to the transposed representation and take the orthogonal complements.
This can be done in time
\begin{align}\label{eq:Tout}
    O\left(\sum_{i \in \Nout_0}p(i)\alpha(i)\min\{p(i), \alpha(i)\} + |\alpha| \bigg(\sum_{i \in N_0} \alpha(i)^3 + \sum_{a \in N_1} \alpha(ha)\alpha(ta)(\alpha(ha) + \alpha(ta))\bigg)\right).
\end{align}
Once we have the subrepresentations $\Wc$ and $\Wno$, we can find the Kalman-type decomposition by directly applying Theorem~\ref{thm:Kalman-decomp-NTI}.
Summarizing this discussion, we have the following theorem.
Let $\Tc(\alpha, m)$ and $\To(\alpha, p)$ denote the time complexity in \eqref{eq:Tin} and \eqref{eq:Tout}, respectively.

\begin{theorem}\label{thm:alg}
    There exist algorithms that, given a networked LTI system, find
    \begin{itemize}
        \item the controllable subrepresentation $\Wc$ in $\Tc(\alpha, m)$ time;
        \item the unobservable subrepresentation $\Wno$ in $\To(\alpha, p)$ time; and
        \item the Kalman-type decomposition in $\Tc(\alpha, m) + \To(\alpha, p)$ time.
    \end{itemize}
    Therefore, one can decide the network-respecting controllability and observability and find the Kalman-type decomposition in time polynomial in $\alpha, m, p$ and the size of $N$.
\end{theorem}

\subsection{Extension to network-respecting target controllability and observability}\label{subsec:target}
Studies of complex networks often consider \emph{target controllability} instead of standard controllability~\cite{Gao2014,Hao2022,Wang2024}.
Let $Z_c \subseteq N_0$ be a set of subsystems.
A networked LTI system is \emph{target controllable} with respect to $Z_c$ if the projected controllability matrix $\Pi_{Z_c} M_c$ is full row rank, where $M_c$ is the controllability matrix \eqref{eq:c-mat} and $\Pi_{Z_c}$ is the coordinate projection matrix onto the state space of subsystems in $Z_c$.
The standard controllability is a special case of the target controllability with $Z_c = N_0$.

Here, we show that our notion of network-respecting controllability naturally extends to the target version even when subsystems have dimensions greater than one.
For a vertex subset $X \subseteq N_0$, let $\ones_X$ be a weight that takes one on $X$ and zero elsewhere.

\begin{definition}[Network-respecting target controllability]
    We say that a networked LTI system \eqref{eq:networked-LTI} is \emph{network-respecting target controllable} with respect to $Z_c$ if $V$ is $\sigma$-semistable, where $\sigma = -\ones_{Z_c}$.
\end{definition}

When $Z_c = N_0$, this definition coincides with network-respecting controllability (Definition~\ref{def:nc}).
By the same argument as in Section~\ref{subsec:networked-LTI}, one can show the following target version of Theorem~\ref{thm:block-uptri-co}.
\begin{theorem}\label{thm:block-uptri-target-co}
    A networked LTI system \eqref{eq:networked-LTI} is not network-respecting target controllable with respect to $Z_c$ if and only if there exist $g \in \GL(N, \alpha; \R)$ and a dimension vector $\beta \leq \alpha$ such that the block decomposition \eqref{eq:block-uptri-co-A}--\eqref{eq:block-uptri-co-B} holds and $\beta(i) < \alpha(i)$ for some target subsystem $i \in Z_c$.
\end{theorem}

This theorem implies that network-respecting target controllability is a necessary condition for ordinary target controllability.
One can also show that network-respecting target controllability with respect to $Z_c$ is equivalent to $\dim \Wc(i) = \alpha(i)$ for all $i \in Z_c$, where $\Wc$ is the controllable subrepresentation.
Therefore, the target controllability can also be decided in polynomial time by Algorithm~\ref{alg:nc}.

One can also define network-respecting target observability in a dual fashion.
To avoid repetition, we omit the details.

\section{On semistability in acyclic quivers with self-loops}\label{sec:reduction}
In this section, we show the algorithmic results of $\sigma$-semistability in quivers with self-loops.
Here, a quiver is called an \emph{acyclic quiver with self-loops} if each vertex has at most one self-loop and it becomes acyclic after removing all self-loops.
For networked LTI systems with acyclic interconnection quiver $N$, the corresponding (marked) quiver $Q$ is an acyclic quiver with self-loops.
In network-respecting (target) controllability and observability, the weight $\sigma$ is either all nonpositive or all nonnegative, so the simple subspace iteration (Algorithm~\ref{alg:nc}) suffices to decide $\sigma$-semistability.
In this section, we consider a generalization where $\sigma$ is either all nonpositive or all nonnegative on vertices with self-loops but may take arbitrary values on vertices without self-loops.
We show a general reduction of $\sigma$-semistability from such a case to the acyclic case.
This reduction yields a polynomial-time algorithm for deciding $\sigma$-semistability using the algorithms by \cite{Iwamasa2025}, which is based on a completely different approach from the subspace iteration in Algorithm~\ref{alg:nc}.

The following is the main result of this section.

\begin{theorem}\label{thm:reduction}
    Let $Q$ be an acyclic quiver with self-loops and $\sigma: Q_0 \to \Z$ be an integer vertex weight on $Q$.
    Assume that $\sigma(i)$ is nonzero and has the same sign for all $i \in Q_0$ with a self-loop.
    Then, the $\sigma$-semistability of a representation $V$ of $Q$ over $\F = \C$ or $\R$ can be decided in time polynomial in the size of $Q$, $V$, and the largest absolute value of $\sigma$.
    Furthermore, one can find a subrepresentation $W \leq V$ over $\F$ that maximizes $\sigma(\dimv W)$ in time polynomial in the size of $Q$, $V$, and the largest absolute value of $\sigma$. 
\end{theorem}

The key component of the proof is to reduce the problem to the acyclic case.
Let $Q$ and $\sigma$ be as in the theorem, and let $\alpha$ be a dimension vector.
We define a new acyclic quiver $\tilde Q = (\tilde Q_0, \tilde Q_1)$ and a new vertex weight $\tilde \sigma: \tilde Q_0 \to \Z$ as follows.
For each vertex $i \in Q_0$ with a self-loop, we place the following gadget in $\tilde Q$:
the gadget consists of two copies $i^+$ and $i^-$ of $i$ and $\alpha(i)$ parallel arcs from $i^+$ to $i^-$. 
For each vertex $i \in Q_0$ without a self-loop, we place a single copy of $i$ in $\tilde Q$.
Abusing the notation, we also define $i^+ = i^- = i$ for this case.
Finally, for each non-loop arc $a$, we place an arc from $(ta)^-$ to $(ha)^+$ in $\tilde Q$.
The new weight $\tilde \sigma$ on $\tilde Q_0$ is defined as follows. 
For each vertex $i \in Q_0$, if $i$ has a self-loop, we set $\tilde \sigma(i^+) = 0$ and $\tilde \sigma(i^-) = \sigma(i)$.
If $i$ has no self-loop, we set $\tilde \sigma(i) = \sigma(i)$.
See Figure~\ref{fig:reduction} for an example of the construction.

Let $V$ be a representation of $Q$ over $\F$ with $\dimv V = \alpha$.
Define a representation $\tilde V$ of $\tilde Q$ over $\F$ as follows.
To the parallel arcs in the gadget for each $i \in Q_0$ with a self-loop $a$, assign matrix powers $V(a)^k$ ($k=0,\dots, \alpha(i)-1$).
To the other non-loop arcs corresponding to $a \in Q_1$, assign $V(a)$.
The following is the key lemma of the reduction.

\begin{lemma}\label{lem:reduction}
    Suppose that $\sigma(i) < 0$ for all $i \in Q_0$ with a self-loop. Then, $V$ is $\sigma$-semistable over $\F$ if and only if $\tilde V$ is $\tilde \sigma$-semistable over $\F$. 
\end{lemma}
\begin{proof}
    \emph{[If part]} 
    Assume that $\tilde V$ is $\tilde \sigma$-semistable.
    Take any subrepresentation $W \leq V$.
    Define $\tilde W(i^+) = \tilde W(i^-) = W(i)$ for each $i \in Q_0$.
    Then, $\tilde W$ is a subrepresentation of $\tilde V$ by construction.
    Furthermore, $\sigma(\dimv W) = \tilde \sigma(\dimv \tilde W) \leq 0$.
    Similarly, $\sigma(\dimv V) = \tilde \sigma(\dimv \tilde V) = 0$.
    Thus, $V$ is $\sigma$-semistable.

    \emph{[Only if part]}
    Assume that $V$ is $\sigma$-semistable.
    Take any subrepresentation $\tilde W \leq \tilde V$ with $\tilde \sigma(\dimv \tilde W)$ maximum.
    For any self-loop $a$ at $i$, we have $\sum_{k=0}^{\alpha(i)-1} V(a)^k \tilde W(i^+) \leq \tilde W(i^-)$ by the construction of the gadget.
    Since $\tilde \sigma(i^-) = \sigma (i) < 0$ by the assumption and $\tilde \sigma(\dimv \tilde W)$ is maximum, we have 
    \[
        \tilde W(i^-) = \sum_{k=0}^{\alpha(i)-1} V(a)^k \tilde W(i^+) = \sum_{k=0}^\infty V(a)^k \tilde W(i^+),
    \]
    where the second equality follows from the Cayley-Hamilton theorem.
    Hence $\tilde W(i^-)$ is a $V(a)$-invariant subspace, i.e., $V(a) \tilde W(i^-) \leq \tilde W(i^-)$.
    For $\tilde a \in \tilde Q_1$ coming from a non-loop arc $a \in Q_1$, we have $V(a) \tilde W((t\tilde a)^-) \leq \tilde W((h\tilde a)^+)$ by the construction of $\tilde Q$.
    These together imply that $\tilde W(i^-)$ ($i \in Q_0$) form a subrepresentation of $V$.
    Furthermore, $\tilde \sigma(\dimv \tilde W) = \sum_{i \in Q_0} \sigma(i)\dim \tilde W(i^-) \leq 0$, where the inequality follows from the $\sigma$-semistability of $V$.
    Similarly, $\tilde \sigma(\dimv \tilde V) = \sigma(\dimv V) = 0$.
    Therefore, $\tilde V$ is $\tilde \sigma$-semistable.
\end{proof}

\begin{figure}
    \centering
    \subcaptionbox{$Q, V, \sigma$}{\includegraphics[height=4cm]{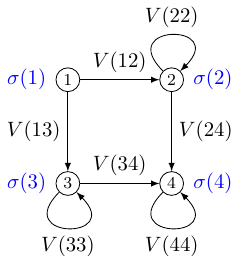}}
    \hspace{3em}
    \subcaptionbox{$\tilde Q, \tilde V, \tilde \sigma$. The values of $\tilde \sigma$ are attached to each vertex (in \textcolor{blue}{blue}).}{\includegraphics[height=4cm]{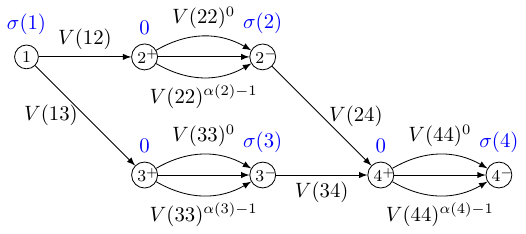}}
    \caption{An example of reduction from $(Q, V, \sigma)$ to $(\tilde Q, \tilde V, \tilde \sigma)$. The value of $\sigma(1)$ may be arbitrary, whereas the values of $\sigma(2)$, $\sigma(3)$, and $\sigma(4)$ are nonzero and have the same sign.}\label{fig:reduction}
\end{figure}

We are now ready to prove Theorem~\ref{thm:reduction}.

\begin{proof}[Proof of Theorem~\ref{thm:reduction}]
    First, we consider the case where $\sigma(i) < 0$ for all $i \in Q_0$ with a self-loop.
    Then, by Lemma~\ref{lem:reduction}, the $\sigma$-semistability of $V$ can be decided by checking the $\tilde \sigma$-semistability of $\tilde V$.
    Since $\tilde Q$ is acyclic, one can apply the algorithm of \cite{Iwamasa2025} for $\tilde V$ and $\tilde \sigma$.

    Furthermore, the above proof of Lemma~\ref{lem:reduction} shows that, given a subrepresentation $\tilde W \leq \tilde V$ that maximizes $\tilde \sigma(\dimv \tilde W)$, one can construct a subrepresentation $W \leq V$ such that $\sigma(\dimv W) = \tilde \sigma(\dimv \tilde W)$, and vice versa.
    Therefore, a subrepresentation $W \leq V$ that maximizes $\sigma(\dimv W)$ can be found by applying the algorithm of \cite{Iwamasa2025} for $\tilde V$ and $\tilde \sigma$.
    
    If $\sigma(i) > 0$ for all $i \in Q_0$ with a self-loop, apply the same argument to the transposed representation $V^\top$ of the transposed quiver $Q^\top$ and the weight $-\sigma$.
\end{proof}

\subsection*{Acknowledgments}
KM was supported by JSPS/MEXT KAKENHI JP23K11001
and by JST ERATO Grant Number JPMJER2301, Japan.
TS was supported by JSPS KAKENHI Grant Number JP19K20212, and JST, PRESTO Grant Number JPMJPR24K5, Japan.

\subsection*{Declaration of generative AI and AI-assisted technologies in the manuscript preparation process}
During the preparation of this work, the authors used GPT-5.6 Pro to proofread drafts written by the authors.
In Section~\ref{subsec:networked-LTI-alg} of an early draft, we described a different algorithm based on the reduction in Section~\ref{sec:reduction} and the existing algorithm for acyclic quivers~\cite{Iwamasa2025}.
As a result, our initial algorithm only worked when $N$ is acyclic.
After we provided the draft to GPT-5.6 Pro for proofreading, it suggested that we could remove this assumption and even provided a simpler algorithm along with our arguments.
We carefully reviewed the suggested algorithm and found that it is indeed correct.
We decided to include the suggested algorithm as Algorithm~\ref{alg:nc} in the present manuscript.
After using this tool/service, the authors reviewed and edited the content as needed and take full responsibility for the content of the published article.

\bibliographystyle{alpha}
\bibliography{main}

\appendix

\section{On semistability over the real and complex fields}\label{sec:ss-R}
In this section, we discuss the relationship between $\sigma$-semistability over the real and complex fields.
We emphasize that here we assume King's criterion (Lemma~\ref{lem:king}) only for $\F = \C$.
Let $V$ be a real representation of a quiver $Q$.
Naturally, $V$ can be seen as a complex representation of $Q$ as well, which we denote by $V_\C$, that is, $V_\C = V \otimes_\R \C$.

\begin{lemma}
    For a real representation $V$ of a quiver $Q$ with a dimension vector $\alpha$ and a weight $\sigma$ on $Q$, the following are equivalent:
    \begin{enumerate}
        \item $V$ is $\sigma$-semistable over $\R$.
        \item $\sigma(\dimv V) = 0$ and $\sigma(\dimv W) \leq 0$ for each real subrepresentation $W \leq V$.
        \item $V_\C$ is $\sigma$-semistable over $\C$.
    \end{enumerate}
\end{lemma}
\begin{proof}
    \emph{[1 $\implies$ 2]} We will show the contrapositive.
    Suppose that $\sigma(\dimv V) \neq 0$. 
    Consider $g_i = t I_{\alpha(i)}$ ($i \in Q_0$) for $t \in \R \setminus \{0\}$. 
    Then, $(g \cdot V)(a) = t V(a) t^{-1} = V(a)$ ($a \in Q_1$) and $\chi_\sigma(g) = \prod_{i \in Q_0} t^{\sigma(i) \alpha(i)} = t^{\sigma(\dimv V)}$.
    Therefore, by taking $t \to 0$ or $t \to \infty$, we have $\chi_\sigma(g) \to 0$, while $(g \cdot V) \in \Rep(Q, \alpha; \R)$.
    Hence, $V$ is not $\sigma$-semistable over $\R$.

    Similarly, suppose that there exists some real subrepresentation $W \leq V$ such that $\sigma(\dimv W) > 0$.
    By the change of basis, one can assume that
    \[
        V(a) = \begin{bNiceMatrix}[first-row,first-col]
                  & W(ta) &  \\
            W(ha) & W_{11}(a) & W_{12}(a) \\
                  & O & W_{22}(a)
        \end{bNiceMatrix} 
        \quad (a \in Q_1). 
    \]
    Consider
    \[
        g_i = \begin{bNiceMatrix}[first-row,first-col]
                  & W(i) &  \\
            W(i) & t I_{\dim W(i)} & O \\
                  & O & I_{\alpha(i) - \dim W(i)}
        \end{bNiceMatrix} 
        \quad (i \in Q_0)
    \]
    for $t \in \R \setminus \{0\}$.
    Then, 
    \[
        (g \cdot V)(a) = \begin{bNiceMatrix}[first-row,first-col]
                  & W(ta) &  \\
            W(ha) & W_{11}(a) & t W_{12}(a) \\
                  & O & W_{22}(a)
        \end{bNiceMatrix}
        \quad (a \in Q_1)
    \]
    and $\chi_\sigma(g) = \prod_{i \in Q_0} t^{\sigma(i) \dim W(i)} = t^{\sigma(\dimv W)}$.
    Therefore, by taking $t \to 0$, we have $\chi_\sigma(g) \to 0$ (since $\sigma(\dimv W) > 0$), while $(g \cdot V) \in \Rep(Q, \alpha; \R)$.
    Hence, $V$ is not $\sigma$-semistable over $\R$.

    \emph{[2 $\implies$ 3]} 
    Again, we show the contrapositive.
    Suppose that $V_\C$ is not $\sigma$-semistable over $\C$.
    By King's criterion over $\C$ (Lemma~\ref{lem:king} for $\F = \C$), $\sigma(\dimv V_\C) \neq 0$ or there exists some complex subrepresentation $W \leq V_\C$ such that $\sigma(\dimv W) > 0$.
    If $\sigma(\dimv V_\C) \neq 0$, then $\sigma(\dimv V) = \sigma(\dimv V_\C) \neq 0$.
    Therefore, we can assume that $\sigma(\dimv V_\C) = 0$. 
    Take a complex subrepresentation $W \leq V_\C$ with $\sigma(\dimv W) > 0$.
    Since $V$ is a real representation, the complex conjugates $\bar W(i)$ ($i \in Q_0$) also form a complex subrepresentation $\bar W$ of $V_\C$ with $\sigma(\dimv \bar W) > 0$.
    Thus, $W_1 := W \cap \bar W$ and $W_2 := W + \bar W$ are subrepresentations of $V_\C$.
    Also note that $W_1$ and $W_2$ are invariant under complex conjugation, i.e., $W_1 = \bar W_1$ and $W_2 = \bar W_2$.
    By Proposition~\ref{prop:complexification-subsp} below, each $W_k(i)$ ($k = 1,2$) is the complexification of some real subspace $U_k(i) \leq V(i)$.
    Obviously, the subspaces $U_k(i)$ form a real subrepresentation $U_k$ of $V$.
    Since
    \[
        0 < \sigma(\dimv W) + \sigma(\dimv \bar W) = \sigma(\dimv (W \cap \bar W)) + \sigma(\dimv (W + \bar W)),
    \]
    at least one of $\sigma(\dimv W_1)$ and $\sigma(\dimv W_2)$ is positive.
    Then, the corresponding real subrepresentation $U_1$ or $U_2$ satisfies $\sigma(\dimv U_1) > 0$ or $\sigma(\dimv U_2) > 0$.

    \emph{[3 $\implies$ 1]} 
    If $V_\C$ is $\sigma$-semistable over $\C$, the orbit closure of $(V_\C, 1) \in \Rep(Q, \alpha; \C) \oplus \chi_\sigma$ under the $\GL(Q, \alpha; \C)$-action does not intersect with $\Rep(Q, \alpha; \C) \times \{0\}$.
    Since $\GL(Q, \alpha; \R)$ is a subgroup of $\GL(Q, \alpha; \C)$, the orbit closure of $(V, 1)$ under the $\GL(Q, \alpha; \R)$-action does not intersect with $\Rep(Q, \alpha; \C) \times \{0\}$.
    Because $V$ is real, every element in the latter orbit is real, so it does not intersect with $\Rep(Q, \alpha; \R) \times \{0\}$ either.
    Thus, $V$ is $\sigma$-semistable over $\R$.
\end{proof}

\begin{proposition}\label{prop:complexification-subsp}
    Let $W \leq \C^n$ be a complex subspace.
    If $W = \bar W$, then there exists a real subspace $U \leq \R^n$ such that $W$ is the complexification of $U$, i.e., $W = U \otimes_\R \C$.
\end{proposition}
\begin{proof}
    Let $U \leq \R^n$ be a real subspace such that $U \otimes_\R \C \leq W$ and $\dim U$ is maximum among all such real subspaces.
    Let $U_\C = U \otimes_\R \C$ and suppose that $U_\C \neq W$.
    Then, there exists a nonzero $w \in W \setminus U_\C$.
    Since $W = \bar W$, the complex conjugate $\bar w$ is also in $W$.
    Therefore, $\Re(w) = \frac{1}{2}(w + \bar w)$ is a real vector in $W$.
    Similarly, $\Im(w) = \frac{1}{2\sqrt{-1}}(w - \bar w)$ is also a real vector in $W$.
    By the maximality of $U$, we must have $\Re(w), \Im(w) \in U$; otherwise, we can extend $U$ by adding these real vectors.
    Then, $w = \Re(w) + \sqrt{-1} \Im(w) \in U_\C$, a contradiction to $w \in W \setminus U_\C$.
\end{proof}

\end{document}